\documentclass[reqno, 11pt]{amsart}

\usepackage{harpoon}
\usepackage{amsmath,amsthm,amssymb}
\usepackage{mathbbol} 
\usepackage{xfrac}
\usepackage{faktor}
\usepackage{mathrsfs}
\usepackage{amssymb,amsmath,amsxtra,graphicx,amsfonts,bm}
\usepackage{amsthm}
\usepackage{times}
\usepackage{bbold}
\usepackage{stmaryrd}
\usepackage{array}
\usepackage{mathrsfs}
\usepackage{dsfont}
\usepackage{url}
\usepackage{tikz}
\usepackage{tikz-cd}
\usepackage[all]{xy}
\usepackage{graphicx}
\usetikzlibrary{shapes,fit,arrows,calc,backgrounds,positioning,matrix,tqft}
\usepackage[T1]{fontenc}

\theoremstyle{plain}
\newtheorem{thm}{Theorem}
\newtheorem{lem}[thm]{Lemma}
\newtheorem{cor}[thm]{Corollary}

\theoremstyle{definition}

\newtheorem*{claim*}{Claim}

\newtheorem*{ack}{Acknowledgment}
\newtheorem*{notat}{Notations and conventions}

\theoremstyle{remark}

\newcommand{\ZZ}{\mathbb{Z}}

\newcommand{\ot}{\otimes}
\newcommand{\ch}{\operatorname{char}}

\newcommand{\onto}{\twoheadrightarrow}

\newcommand{\longiso}{\overset{\text{\raisebox{-.5ex}{$\sim$}}}{\longrightarrow}}

\renewcommand{\k}{\mathbb{k}}

\newcommand{\gen}[1]{\langle{#1}\rangle}
\renewcommand{\bar}[1]{\overline{#1}}

\newcommand{\Hom}{\operatorname{Hom}}

\renewcommand{\d}{\delta}
\newcommand{\D}{{\rm \Delta}}

\newcommand{\e}{\epsilon}

\newcommand{\sP}{\mathscr{P}}

\DeclareMathOperator{\Char}{char}

\newcommand{\fg}{\mathfrak{g}}

\newcommand{\bdot}{\,\text{\raisebox{-.45ex}{$\boldsymbol{\cdot}$}}\,}

\newcommand{\supp}{\operatorname{supp}}

\newcommand{\gr}{\mathrm{gr}}

\newcommand{\bfn}{\mathbf{n}}
\newcommand{\bfm}{\mathbf{m}}
\newcommand{\bfs}{\mathbf{s}}
\newcommand{\bfr}{\mathbf{r}}
\newcommand{\bff}{\mathbf{f}}
\newcommand{\bfg}{\mathbf{g}}

\newcommand{\bfi}{\mathbf{i}}
\newcommand{\bfj}{\mathbf{j}}
\newcommand{\bfl}{\mathbf{l}}
\newcommand{\bfk}{\mathbf{k}}

\newcommand{\bfzero}{\mathbf{0}}

\begin{document}

\title[On Actions of Connected Hopf Algebras]%
{On Actions of Connected Hopf Algebras}

\author{Ramy Yammine}

\address{Department of Mathematics, Temple University, Philadelphia, PA 19122, USA}

\maketitle

\begin{abstract}
Let $H$ be a connected Hopf algebra acting on an algebra $A$. Working over a base field having
characteristic
$0$, we show that for a given prime (semi-prime, completely prime) ideal $I$ of $A$, 
the largest $H$-stable ideal of A contained in $I$ is also prime (semi-prime, completely prime).
We also prove a similar result for certain subrings of
convolution algebras.
\end{abstract}

\maketitle


\section{Introduction}
\label{S:Intro}

\subsection{}
\label{SS:Intro1}
This paper investigates certain aspects of actions of Hopf algebras on other algebras.
Specifically, let $A$ be an algebra (associative, with $1$) over a field $\k$
and let $H$ be a Hopf $\k$-algebra.
A (left) \emph{$H$-action} on $A$ is a $\k$-linear map
$H \ot_\k A \to A$, $h\ot a \mapsto h.a$, 
that makes $A$ into a 
left $H$-module and satisfies the conditions 
$h.(ab) =  (h_1.a)(h_2.b)$ and $h.1 =  \gen{\e,h} 1$ for all $h\in H$ and $a,b \in A$.
Here, $\Delta \colon H \to H\ot H$, $h \mapsto h_1 \ot h_2$\,,  is the comultiplication 
and $\e \colon H \to \k$ is the counit. 
An algebra $A$ that is equipped with an $H$-action is called an \emph{$H$-module algebra}.
An ideal of $A$ that is also an $H$-submodule 
is referred to as an \emph{$H$-ideal}. 
The sum of all $H$-ideals of $A$ that are
contained in a given arbitrary ideal $I$, clearly the unique maximal $H$-ideal of $A$ 
that is contained in $I$, is called the \emph{$H$-core} of $I$; it will be 
denoted by $(I:H)$. Explicitly,
\[
(I:H) = \{ a \in A \mid h.a \in I \text{ for all } h \in H \}. 
\]

\subsection{}
\label{SS:Intro2}
Our main concern is with the transfer of properties from $I$ to $(I:H)$ under the assumption
that $\ch\k = 0$ and $H$ is \emph{connected}, that is, $H$ has no simple subcoalgebras other than $\k = \k 1$. 
This class of Hopf algebras
includes all enveloping algebras of Lie algebras. In fact, if $\ch\k = 0$, then
every cocommutative connected Hopf algebra is an enveloping algebra \cite[\S\S 5.5, 5.6]{sM93}.
In this setting, the theorem below is known \cite[3.3.2 and 3.8.8]{jD96}. Our contribution consists
in removing the cocommutativity assumption.

\begin{thm}
\label{Thm}
Let $H$ be a connected Hopf algebra over a field of characteristic $0$, let $A$ be an
$H$-module algebra, and let $I$ be an ideal of $A$. If $I$ is prime (semi-prime, completely prime), then so is $(I:H)$.
\end{thm}

We remark that a result analogous to Theorem~\ref{Thm} is also known for rational actions of 
connected affine algebraic groups 
over algebraically closed base fields \cite[Corollary 1.3]{wC92}, \cite[Proposition 19]{mL08}.

\subsection{}
\label{SS:Intro3}
Our main tool in the proof of Theorem~\ref{Thm}
is an analysis of the coradical filtration of $H$. This will allow us to construct 
a certain PBW basis of $H$, which 
is then used to prove that certain subrings of
convolution algebras of $H$ are prime (semiprime, a domain); see Theorem~\ref{T:Tech3} below.
Theorem~\ref{Thm} will be derived from this result.

\begin{notat}
The notation introduced in the foregoing will remain in effect below.
We work over a base field $\k$ and will write $\ot = \ot_\k$\,.
Throughout, $H$ will denote a Hopf $\k$-algebra. 
All further assumptions on $\k$ and $H$ will be specified as they are needed. 
We denote by $\ZZ_+$ the set of non-negative integers and by $\ZZ_{>0}$
the set of strictly positive integers.
\end{notat}


\section{Set-theoretic preliminaries}

We introduce here a certain monoid that will be essential for the proof of Theorem \ref{Thm}.

\subsection{The monoid $M$}\label{monoidM}

Let $\Lambda$ be a set. 
We consider functions $\bfm \colon \Lambda \longrightarrow \ZZ_{+}$ whose support 
$\supp\bfm:= \{ \lambda \in \Lambda \mid \bfm(\lambda) \neq 0 \}$ is finite.  
Using the familiar ``pointwise'' addition of functions, $(\bfm+\bfm')(\lambda) = \bfm(\lambda) + \bfm'(\lambda)$
for $\lambda \in \Lambda$, we obtain a commutative monoid,
\[ 
M := \ZZ_{+}^{(\Lambda)} = \{ \bfm \colon \Lambda \longrightarrow \ZZ_{+} \mid \supp\bfm \text{ is finite}\}, 
\]
The identity element is the unique function $\bfzero \in M$ 
such that $\supp\bfzero = \varnothing$; so 
$\bfzero(\lambda) = 0$ for all $\lambda \in \Lambda$. We will think of $\Lambda$ as a subset 
of $M \setminus \{ \bfzero\}$, 
identifying $\lambda \in \Lambda$ with the function $\delta_\lambda \in M$ that is defined by
\[
\delta_\lambda(\lambda') = \delta_{\lambda,\lambda'} \qquad (\lambda, \lambda' \in \Lambda), 
\]
where $\delta_{\lambda,\lambda'}$ is the Kronecker delta.
Thus, $\supp \d_\lambda = \{ \lambda \}$
and each $\bfm \in M$ has the form $\bfm = \sum_{\lambda \in \supp\bfm} \bfm(\lambda) \d_\lambda$\,.

Now assume that $\Lambda$ is equipped with a map $|\bdot| : \Lambda \longrightarrow \ZZ_{> 0}$\,,
which we will think of  as a ``degree.'' We extend $|\bdot|$ from $\Lambda$ to $M$,
defining the degree of an element $\bfm \in M$ by 
\[
|\bfm| := \sum\limits_{\lambda \in \supp\bfm}\bfm(\lambda)|\lambda| \in \ZZ_+\,.
\] 
Note that $|\delta_\lambda| = |\lambda|$ and that
degrees are additive: $|\bfm+ \bfm'| = |\bfm| + |\bfm'|$ for $\bfm,\bfm' \in M$.
For a given $d \in \ZZ_+$\,, we put
\[ 
M_{d} := \{ \bfm \in M \mid |\bfm| = d \} \qquad \text{and} \qquad \Lambda_d := \Lambda \cap M_d \,.
\]
Thus,  $M = \bigsqcup_{d \in \ZZ_+} M_d$ and 
$M_0 = \{ \bfzero \}$, $\Lambda_0 = \varnothing$.

\subsection{A well-order on $M$}
\label{orderM}

Now assume that $\Lambda$ is equipped with a total order $\le$ such that 
$\Lambda_1 < \Lambda_2 < \dots $\,, 
that is, elements of $\Lambda_i$ precede elements of $\Lambda_j$ in this order if $i < j$.
We extend $\le$ to $M$ as follows. First order elements of $M$ by degree, i.e., 
\[
\{ \bfzero \} = M_0 < M_1 < M_2 < \dots\,.
\]
For the tie-breaker, let $\bfn \neq \bfm \in M_d$ and put 
\begin{equation*}
\mu = \mu_{\bfn,\bfm}: = \max\{\lambda \in \Lambda \mid \bfn(\lambda) \neq \bfm(\lambda)\}.
\end{equation*}
Note that
$\{\lambda \in \Lambda \mid \bfn(\lambda) \neq \bfm(\lambda)\}$  is 
finite, being contained in $\supp(\bfn) \cup \supp(\bfm)$. 
If $\bfn(\mu) > \bfm(\mu)$, then we define $\bfn > \bfm$.
For $\lambda \neq \lambda' \in \Lambda$, this becomes $\delta_\lambda >  \delta_{\lambda'}$ 
if and only if $\lambda > \lambda'$. 

\begin{lem}
\label{OrderM}
The above order on $M$ has the following properties:
\begin{enumerate}
    \item $\le$ is a total order on $M$ and $\bfzero$ is the unique minimal element of $M$.
    \item If $\bfn,\bfm \in M$ are such that $\bfn<\bfm$, then $\bfn+\bfr < \bfm+\bfr$ for every $\bfr\in M$.
    \item If the order of each $\Lambda_d$ is a well-order, then $\leq$ is a well-order of $M$.
\end{enumerate}
\end{lem}

\begin{proof} 
(a)
It is clear that $\bfzero < \bfm$ for any $\bfzero \neq \bfm \in M$ and that exactly one of
$\bfn = \bfm$, $\bfn < \bfm$ or $\bfn > \bfm$ holds for any two $\bfn,\bfm \in M$. 
To check transitivity, consider elements $\bfn < \bfm < \bfr$ of $M$. Then $\bfn \neq \bfr$
and we need to show that $\bfn < \bfr$. For this, we may assume that
$|\bfn| = |\bfm| = |\bfr|$, because otherwise $|\bfn| < |\bfr|$ and we are done.
Put $\mu = \mu_{\bfn,\bfm}$,
$\nu = \mu_{\bfm,\bfr}$ and
$\rho = \mu_{\bfn,\bfr}$. Then $\bfn(\mu) < \bfm(\mu)$,
$\bfm(\nu) < \bfr(\nu)$ and we need to show that $\bfn(\rho) < \bfr(\rho)$.
Note that $\bfn(\rho) \neq \bfr(\rho)$ iff $\bfn(\rho) \neq \bfm(\rho)$ or $\bfm(\rho) \neq \bfr(\rho)$, and therefore $\rho = \max\{\mu,  \nu\}$.
If $\rho = \mu$ we get $\bfn(\rho) < \bfm(\rho) \leq \bfr(\rho)$ and if $\rho = \nu$ we get $\bfn(\rho)\leq \bfm(\rho) < \bfr(\rho)$, both of which give us what we need.

\medskip
  
(b)
By additivity of degrees,  $\bfn+\bfr < \bfm+\bfr$ certainly holds if  $|\bfn| < |\bfm|$. Assume that 
$\bfn < \bfm$ but $|\bfn| = |\bfm|$ and let
$\mu = \mu_{\bfn,\bfm}$; so $\bfn(\mu) < \bfm(\mu)$.
Then we also have $\mu = \mu_{\bfn + \bfr,\bfm + \bfr}$
and $(\bfn+\bfr)(\mu) < (\bfm+\bfr)(\mu)$.

\medskip
    
(c)
Our assumption easily implies that the order of $\Lambda$ is a well-order.
Now let $\varnothing \neq S\subseteq M$ be arbitrary. We wish to find a minimal element in $S$.
Put
\[
d = d(S):= \min \{ d' \in \ZZ_{+} \mid S\cap M_{d'} \neq \varnothing \}. 
\]

If $d = 0$, then $\bfzero \in S$ and $\bfzero$ is the desired minimal element by (a).
So assume $d>0$ and that all 
$\varnothing \neq T \subseteq M$ with $d(T) < d$ have a minimal element. 
The desired minimal element of $S$
must belong to $S \cap M_d$\,; so we may assume without loss that $S \subseteq M_d$\,.
Write $\mu_{\bfn} = \mu_{\bfn,\bfzero} = \max(\supp(\bfn))$ for $\bfzero \neq \bfn \in M$ and put 
\[
\lambda_{S} :=\min\{ \mu_{\bfn} \mid \bfn \in S \} \in \Lambda;
\] 
this is well defined since we have a well-order on $\Lambda$. Consider the subsets
\[
        S' := \{ \bfs \in S \mid \mu_{\bfs} = \lambda_{S}\} \quad \text{and} \quad 
        S'' := \{ \bfs\in S' \mid \bfs(\lambda_{S}) \le \bfs'(\lambda_S) \ \forall \bfs' \in S^{'}\}
\]
and put $z_{S} := \bfs(\lambda_{S})$ for any $\bfs \in S''$; so $z_S \in \ZZ_{> 0}$\,. 
Since all elements of $S''$ are smaller than elements of $S \setminus S''$, 
it suffices to find a minimal element in $S''$.
Notice that, for $\bfr\in S'$, the function $\bfr-z_{S}\delta_{\lambda_{S}}$ belongs to $M$. 
Indeed, $\bfr(\lambda)\geq z_{S}$ for all $\lambda\in \Lambda$ and hence 
$(\bfr-z_{S}\delta_{\lambda_{S}})(\lambda) = \bfr(\lambda) - z_{S} \in \ZZ_{+}$.
By part (b), comparing two elements 
$\bfr, \textbf{t}\in S^{''}$ is equivalent to comparing $\bfr-z_{S}\delta_{\lambda_{S}}$ and $\textbf{t}-z_{S}\delta_{\lambda_{S}}$.
Thus setting $T:= \{\bfs - z_{S}\delta_{\lambda_{S}} \mid \textbf{s}\in S''\}$, our goal is to show that $T$ has a minimal element. 
But $T \subset M$ and all elements of $T$ are of the form $\bfr-z_{S}\delta_{\lambda_{S}}$ where $\bfr\in S''$, so the degree of any element in 
$T$ is $d - |\lambda_{S}| z_S < d$. 
Therefore, $d(T) < d$ and, by our inductive hypothesis, $T$ has a minimal element. This finishes the proof.
\end{proof}


\section{The coradical filtration}
\label{S:corad}

In this section, we recall some standard definitions and state some known facts for later use.

\subsection{Coradically graded coalgebras}
\label{SS:corad}

Let $C$ be a $\k$-coalgebra. 
Throughout, we let $C_0$ denote the \emph{coradical}, that is, the sum of 
all its simple sub-coalgebras of $C$.  
The coalgebra $C$ is said to 
be \emph{connected} if $C_0 = \k$.
The \emph{coradical filtration} of $C$ is recursively  defined by 
\[
C_{j+1}:= \{ x\in C \hspace{0.1in} | \hspace{0.1in} \D(x) \in C_j\ot C + C\ot C_0 \} \qquad (j\geq 0).
\]
We also put $C_{-1}:= \{0\}$.
This yields a coalgebra filtration \cite[Prop. 4.1.5]{dR12}:  
\[
\{0\} = C_{-1}\subseteq C_0 \subseteq \dots \subseteq C_j \subseteq C_{j+1} \subseteq 
\dots \subseteq C= \bigcup_{n\geq 0} C_n
\]
and
\[ 
\D(C_j) \subseteq \sum_{0 \leq i \leq j}C_i \ot C_{j-i} \qquad (j \ge 0).
\]
We consider the graded vector space that is associated to this filtration:
\[
\gr\, C :=  \bigoplus_{n\geq 0}{(\gr\, C)(n)} \quad \text{with} \quad 
(\gr\, C)(n) := {C_n}/{C_{n-1}} \,.
\] 
Any element $0 \neq \bar{c} \in (\gr\, C)(n)$ will be called \emph{homogeneous of 
degree} $n$.
The vector space  $\gr\, C$  inherits a natural coalgebra structure from $C$;
this will be discussed greater detail in \ref{SS:comult} below.
The resulting coalgebra $\gr\,C$ is \emph{coradically graded} \cite[Prop.~4.4.15]{dR12}:  
the coradical filtration of  $\gr\, C$ is given by 
\[
(\gr\, C)_n = \bigoplus\limits_{0 \leq k \leq n}(\gr\, C)(k) \qquad (n\in \ZZ_{+}).
\]
In particular, the coradical of $\gr\, C$ coincides with the coradical of $C$: 
\[
(\gr\, C)_0 = (\gr\, C)(0) = C_0\,.
\]
Hence $\gr\, C$ is connected if and only if $C$ is connected.

\subsection{Some facts about connected Hopf algebras}
\label{SS:ConnFacts}

Let $H$ be a connected Hopf $\k$-algebra. By \ref{SS:corad}, the graded coalgebra $\gr\, H$ is connected as well. 
Below, we list some additional known properties of the coradical filtration $(H_n)$ of $H$ and
of $\gr\, H = \bigoplus_{n \ge 0} H_n/H_{n-1}$:

\begin{enumerate}
\item 
$\gr\, H$ is a \emph{coradically graded Hopf algebra} \cite[Prop.~7.9.4]{dR12}: 
The coradical filtration $(H_n)$ is also an algebra filtration, that is, $H_n H_m \subseteq H_{n+m}$ for all $n,m$.
Furthermore, all $H_{n}$ are stable under the antipode of $H$.
Thus, $\gr\, H$ inherits a natural Hopf algebra structure from $H$, graded as a $\k$-algebra and coradically graded
as $\k$-coalgebra.
\item 
$\gr\, H$ is  \emph{commutative} as $\k$-algebra \cite[Prop.~6.4]{gZ13}.
\item 
If $\ch \k = 0$, then $\gr\, H$ is isomorphic
to a \emph{graded polynomial algebra}: There is an isomorphism of graded algebras, 
\[
\gr\, H \cong \k [z_{\lambda}]_{\lambda \in \Lambda}
\]
for some family $\big(z_{\lambda}\big)_{\lambda \in \Lambda}$ of homogeneous algebraically
independent variables,
necessarily of strictly positive degrees  \cite[Proposition 3.6]{dLySgZxx}.
\end{enumerate}


\section{A Poincar{\'e}-Birkhoff-Witt basis for $H$}

The classical Poincar{\'e}-Birkhoff-Witt (PBW) Theorem (see, e.g., \cite{jD96}, \cite{mL18}) provides a $\k$-basis
for the enveloping algebra $U\fg$ of any Lie $\k$-algebra $\fg$ that consists of certain
\emph{ordered monomials}:
If $(x_{\lambda})_{\lambda\in \Lambda}$ is a $\k$-basis of $\fg$ and $\leq$ is a total order on $\Lambda$,
then a $\k$-basis of $U\fg$ is given by the monomials 
$x_{\lambda_1}x_{\lambda_2}\dots x_{\lambda_{n}}$ with $n\in \ZZ_{+}$ and 
$\lambda_{1}\leq \lambda_{2} \leq \dots \leq \lambda_{n}$\,.
In this section, working in characteristic $0$,
we will construct an analogous basis for an arbitrary connected Hopf algebra $H$, 
formed by ordered monomials in specified algebra generators for $H$. We will call it a PBW basis of $H$.

\medskip

\textit{We assume from this point onward that $H$ is a connected Hopf $\k$-algebra and that $\Char \k = 0$.}

\subsection{PBW basis}
\label{SS:PBW}

We fix a family of homogeneous algebra generators $(z_{\lambda})_{\lambda \in \Lambda}$ of $\gr\, H$ as in 
\ref{SS:ConnFacts}(c), viewing the isomorphism $\gr\, H \cong \k [z_{\lambda}]_{\lambda \in \Lambda}$ 
as an identification. Define a map $|\bdot| \colon \Lambda \rightarrow \ZZ_{> 0}$ by 
\[
|\lambda| := \deg{z_{\lambda}}\,.
\]
Fix a well-order on each set $\Lambda_d := \{ \lambda \in \Lambda \mid |\lambda| = d \}$.
As in \ref{monoidM}, \ref{orderM} consider the monoid
$M = \{ f \colon \Lambda \longrightarrow \ZZ_{+} \hspace{0.1in} | \hspace{0.1in} \supp(f) \text{ is finite}\}$,
equipped with the well-order $\leq$ coming from the chosen orders on each $\Lambda_d$\,. 

We consider the following 
$\k$-basis of $\gr\, H$: 
\begin{equation}
\label{E:barEn}
z_{\bfn} := \prod_{\lambda \in \Lambda}^{}\frac{1}{\bfn(\lambda)!} z_{\lambda}^{\bfn(\lambda)}
= \prod_{\lambda \in \supp\bfn}^{}\frac{1}{\bfn(\lambda)!} z_{\lambda}^{\bfn(\lambda)}
\qquad (\bfn \in M).
\end{equation}
Note that $z_{\bfn}$ is homogeneous of degree $|\bfn| = \sum_{\lambda \in \Lambda}\bfn(\lambda)|\lambda|$. 
Our goal is to lift this basis to obtain a PBW basis of $H$. 
In detail, for every positive integer $k$, let 
\[
\pi_{k} \colon  H_{k} \twoheadrightarrow (\gr\, H)(k)= {H_k}/{H_{k-1}}
\]
be the 
canonical epimorphism. For each $\lambda \in \Lambda$, let $e_{\lambda}\in H_{|\lambda|}\subseteq H$ 
be a fixed pre-image of $z_{\lambda}$ under $\pi_{|\lambda|}$. We put 
\begin{equation}
\label{E:En}
e_{\bfn} = \prod_{\lambda\in \supp\bfn}^{<}\frac{1}{\bfn(\lambda)!} e_{\lambda}^{\bfn(\lambda)} \qquad (\bfn\in M),
\end{equation}
where the superscript $<$ indicates that the factors occur in the order of increasing $\lambda$. 
Note that $e_{\bfn} \in H_{|\bfn|}$ and $\pi_{|\bfn|}(e_{\bfn}) = z_{\bfn}$\,. Furthermore, $e_{\bfzero} = 1$.

\begin{lem}
\label{BasisM}
\begin{enumerate}
\item
$\mathscr{M}_n := \big(e_{\bfn}\big)_{\bfn \in M, |\bfn| \le n}$ is
a basis of $H_n$  $(n \ge 0)$. Hence, $\mathscr{M} := \big(e_{\bfn}\big)_{\bfn \in M}$ is a basis of $H$.
\item
$e_{\bfn}e_{\bfm} - c_{\bfn,\bfm} e_{\bfn+\bfm} \in H_{|\bfn + \bfm|-1}$
for some $c_{\bfn,\bfm} \in \k^\times$.
\end{enumerate}
\end{lem}
    
\begin{proof}
(a)
We first prove linear independence of $\mathscr{M}$. 
Suppose that there is a linear relation $\sum_{i=1}^{r}a_{i}e_{\bfn_{i}} = 0$ with $r\geq 1$,
$\bfn_{i}\in M$, and $0 \neq a_i \in \k$ for all $i$.  Put $d = \max \{|\bfn_{i}|\}_{1\leq i \leq r}$. Then 
$\sum_{i=1}^{r}a_{i}e_{\bfn_{i}}\in H_d$ and $\pi_{d}(\sum_{i=1}^{r}a_{i}e_{\bfn_{i}}) 
= \sum_{|\bfn_{i}| = d}a_{i}z_{\bfn_{i}} = 0$. 
Since $\{z_{\bfn}\}_{\bfn\in M}$ are linearly independent in $\gr\, H$, it follows that $a_{i} = 0$ for all 
$i\in \{1,2,...,r\}$ such that $|\bfn_{i}| = d$, contradicting our assumption that all $a_{i} \neq 0$.

Next, we show that $\mathscr{M}_n$ spans $H_n$\,. 
For $n = 0$, this is certainly true because $H_{0} = \k$ and $e_{\bfzero}=1$. 
Now let $n > 0$. For any $h\in H_{n}$, the projection $\pi_{n}(h) \in (\gr\, H)(n)$ can be written as a 
$\k$-linear combination $\pi_{n}(h) = \sum_{i=1}^{r}a_{i}z_{\bfn_{i}}$ with $|\bfn_i| = n$. Thus, 
$h - \sum_{i=1}^{r}a_{i}e_{\bfn_{i}} \in H_{n-1}$. Since $H_{n-1}$ is spanned by $\mathscr{M}_{n-1}$ by induction, 
we deduce that $h$ is in the span of $\mathscr{M}_n$\,.

\medskip

(b)
Note that $e_{\bfn+\bfm} \in H_{|\bfn+\bfm|}$ and also
$e_{\bfn}e_{\bfm} \in H_{|\bfn|}H_{|\bfm|} \subseteq H_{|\bfn + \bfm|}$, where the last inclusion
holds by \ref{SS:ConnFacts}(a) and additivity of $|\bdot|$. 
Furthermore, since $\gr\,H$ is commutative by \ref{SS:ConnFacts}(b), we
have $z_{\bfn}z_{\bfm} = c_{\bfn,\bfm} z_{\bfn+\bfm}$ with $c_{\bfn,\bfm} 
= \prod_{\lambda\in \Lambda}\frac{(\bfn+\bfm)(\lambda)!}{\bfn(\lambda)!\bfm(\lambda)!}\in \k^\times$. Therefore,
by definition of the multiplication of $\gr\,H$,
\[
\pi_{|\bfn + \bfm|}(e_{\bfn}e_{\bfm}) = \pi_{|\bfn|}(e_{\bfn})\pi_{|\bfm|}(e_{\bfm}) 
= z_{\bfn}z_{\bfm} = c_{\bfn,\bfm} z_{\bfn+\bfm} = c_{\bfn,\bfm} \pi_{|\bfn + \bfm|}(e_{\bfn+\bfm}) 
\]
and so $e_{\bfn}e_{\bfm} - c_{\bfn,\bfm} e_{\bfn+\bfm} \in H_{|\bfn + \bfm|-1}$ as claimed.
\end{proof}

The family $\mathscr{M}$ will be our PBW basis for $H$.

\subsection{The comultiplication on the PBW basis}
\label{SS:comult}
We will first prove a result concerning the comultiplication of $H$.

\begin{lem}\label{comH}
$\D(h) \in 1\ot h + h\ot 1 + \sum\limits_{i=1}^{n-1}H_i \ot H_{n-i}$ for any $h \in H_n$\,.
\end{lem}

\begin{proof}
We start with some preparations, following the procedure that is described 
in detail in \cite[p.~139ff]{dR12}.
The coradical filtration $(H_n)$ can be used to give $H$ a graded $\k$-vector space structure 
such that $H_n = \bigoplus_{k = 0}^n H(k)$ for all $n \ge 0$ 
and $\epsilon(H(n)) = 0$ for all $n\geq 1$. Then there are isomorphisms 
\[
j_n \colon H(n) \longiso {H_n}/{H_{n-1}} = (\gr\, H)(n) \qquad (n \ge 0)
\]
and we obtain an isomorphism of graded vector spaces,
\[
j := \bigoplus_{n\geq 0} j_n \colon H \longiso \gr\, H\,.
\]
Letting
$p_n \colon H \ot H \onto (H\ot H)(n):= \bigoplus_{r+s = n}H(r)\ot H(s)$
denote the natural projection, we define 
\[
\delta := \bigoplus_{n\geq 0}p_n \circ \D\big| _{H(n)} \colon H \rightarrow H\ot H.
\]
Then, for any $z \in H(n)$ $(n \ge 0)$,
\begin{equation}
\label{E:Dd}
\D(z) - \delta(z) \in \bigoplus_{r < n} (H \ot H)(r) = \bigoplus_{i+j  < n} H(i) \ot H(j).
\end{equation}
Transferring $\delta$ to $\gr\,H$ via $j$, we obtain a comultiplication $\D_{\gr} \colon \gr\,H
\to \gr\,H \ot \gr\,H$ making $\gr\,H$ a graded coalgebra and a commutative diagram:
\begin{equation}
\label{E:jDiagram}
\begin{tikzcd}
H \arrow{r}{\delta} \arrow[swap]{d}{j} & H\ot H \arrow{d}{j\ot j} \\
\gr\, H \arrow{r}{\D_{\gr}} & \gr\, H \ot \gr\, H
\end{tikzcd}
\end{equation}
See \cite[p.~139-141]{dR12} for all this and \cite[Props.~4.4.15, 7.9.4]{dR12} for the fact that the
above construction results in a coradically graded Hopf algebra, $\gr\,H$.

\medskip 

Now let $B$ be a $\k$-bialgebra whose coradical filtration $(B_n)$ is a bialgebra filtration in the sense
of \cite[Definition 5.6.2]{dR12}---we will later take $B = \gr\,H$, where this condition is satisfied. 
For any $n > 0$, we put
\[
\sP_n:= \big\{ b \in B_n \mid \D(z) \in 1\ot b + b \ot 1 +\sum_{i=1}^{n-1}B_i \ot B_{n-i} \big\}.
\]
Clearly, $\sP_n \subseteq \sP_k$ if $k \ge n$ and $\k \sP_n \subseteq \sP_n$ for all $n$. In addition:

\begin{claim*} 
$\sP_n\sP_m \subseteq \sP_{n+m}$ and
$\sP_n + \sP_m \subseteq \sP_{\max (n,m)}$ for all $n,m > 0$.
\end{claim*}

To check this, let $b \in \sP_n$, $c \in \sP_m$\ and put $r:= \max(n,m)$. Then
\[
\begin{aligned}
\D(bc) &= \D(b)\D(c) \\
&\in \big(1 \ot b + b \ot 1 +\sum_{i=1}^{n-1}B_i\ot B_{n-i}\big) 
\big(1 \ot c + c \ot 1 + \sum_{i=1}^{m-1}B_i\ot B_{m-i}\big)\\ 
&\subseteq 1 \ot bc + bc \ot 1 + b\ot c + c\ot b + \sum_{i=1}^{m+n-1}B_i\ot B_{m+n-i} \\
&= 1 \ot bc + bc \ot 1 + \sum_{i=1}^{m+n-1}B_i\ot B_{m+n-i}
\end{aligned}
\]
\[
\begin{aligned}
\D(b+c) &= \D(b)+ \D(c) \\ 
&\in  \big(1\ot b + b\ot 1 +\sum_{i=1}^{n-1}B_i \ot B_{n-i}\big) +  
\big(1\ot c + c\ot 1 +\sum_{i=1}^{m-1}B_i \ot B_{n-i}\big)\\
&\subseteq  1\ot  b +  b \ot 1 + 1\ot  c +  c \ot 1 +\sum_{i=1}^{r-1}B_i \ot B_{r-i}\\
&=1\ot (b + c) + (b + c)\ot 1 + \sum_{i=1}^{r-1}B_i \ot B_{r-i}\,,
\end{aligned}
\]
which proves the Claim.

\medskip

We now apply the foregoing with $B = \gr\,H$ using the notation of \ref{SS:ConnFacts}(c).
By \cite[Theorem 3.1]{dLySgZxx}, we know that $z_{\lambda} \in \sP_{|\lambda|}$
for all $\lambda \in \Lambda$. 
Therefore, the first inclusion in the claim gives that $z_{\bfn} \in \sP_{|\bfn|}$; see \eqref{E:barEn}.
Next, observe that any $z \in (\gr\,H)_n =  (\gr\, H)(0) \oplus \dots \oplus (\gr\, H)(n)$ can be written as 
$z = \sum_{\bfn \in M, |\bfn| \le n} c_{\bfn}z_{\bfn}$\,.
The second inclusion of the claim therefore gives that  $z \in \sP_n$.

To finish, let $h \in H_n$ be given and put $z:= j(h) \in (\gr\,H)_n$\,.
By the foregoing, $z \in \sP_n$ and so 
$\D_{\gr}(z) \in 1 \ot z + z \ot 1 + \sum_{i=1}^{n-1}(\gr\, H)_i\ot(\gr\, H)_{n-i}$\,. 
Diagram \eqref{E:jDiagram} now gives
$\delta(h) \in 1\ot h + h \ot 1 + \sum_{i=1}^{n-1}H_i \ot H_{n-i}$ and
\eqref{E:Dd} further gives 
$\D(h) \in 1\ot h + h \ot 1 + \sum_{i=1}^{n-1} H_i \ot H_{n-i} + \sum_{i=0}^{n-1}H_i\ot H_{n-i-1}$\,.
Finally, note that $\sum_{i=0}^{n-1}H_i\ot H_{n-i-1} \subseteq  \sum_{i=1}^{n-1} H_i \ot H_{n-i}$, because 
$H_0 \ot H_{n-1} \subseteq H_1 \ot H_{n-1}$
and $ H_i\ot H_{n-i-1} \subseteq H_i\ot H_{n-i}$. Thus, 
$\D(h) \in 1\ot h + h \ot 1 + \sum_{i=1}^{n-1} H_i \ot H_{n-i}$, which completes the proof.
\end{proof}

Recall that $\mathscr{M}_{\bfn} = \big(e_{\bfn}\big)_{\bfn \in M, |\bfn|\leq n}$ 
is a basis of $H_{n}$ (Lemma~\ref{BasisM}). For a given $\bfn \in M$, let $H^{<\bfn}$
denote the subspace of $H_{|\bfn|}$
that is spanned by the $e_{\bfi}$ with $\bfi < \bfn$ and put $H^{\le \bfn} = \k e_{\bfn} + H^{<\bfn}$. Furthermore, we put 
$(H\ot H)^{<\bfn} = \sum_{\bfi + \bfj < \bfn} H^{\le\bfi} \ot H^{\le\bfj}$ and
$(H\ot H)^{\le\bfn} = \sum_{\bfi + \bfj \le \bfn} H^{\le\bfi} \ot H^{\le\bfj}$; these are the subspaces of $H \ot H$
that are spanned by the tensors $e_{\bfi} \ot e_{\bfj}$ with $\bfi + \bfj < \bfn$ and $\bfi + \bfj \le \bfn$, respectively

\begin{lem}
\label{Tech1}
For $\bfn \in M$ we have $\D(e_{\bfn}) \in \sum_{\bfm +\bfm' = \bfn} e_{\bfm}\ot e_{\bfm'} + (H\ot H)^{<\bfn}$.
\end{lem}

\begin{proof}
We start with the following auxiliary observation.
Recall from Lemma~\ref{BasisM}(b) that 
$e_{\bfi}e_{\bfj}\in \k e_{\bfi+\bfj} + H_{|\bfi + \bfj|-1}$ and observe that
the right-hand side is contained in $H^{\le \bfi + \bfj}$. It follows that
$H^{\le \bfn} H^{\le\bfm} \subseteq H^{\le \bfn + \bfm}$ for any $\bfn, \bfm \in M$. Therefore,
\begin{equation}
\label{E:prelim}
\begin{aligned}
(H\ot H)^{\leq\bfn}(H\ot H)^{<\bfm} 
&= \sum_{\substack{\bfi + \bfj \le \bfn \\ \bfk+\bfl < \bfm}} H^{\le\bfi}H^{\le\bfk} \ot H^{\le\bfj}H^{\le\bfl} \\
&\subseteq \sum_{\substack{\bfi + \bfj \le \bfn \\ \bfk+\bfl < \bfm}} H^{\le\bfi + \bfk} \ot H^{\le\bfj +\bfl} \\
&\subseteq \sum_{\bfr+\bfs < \bfn + \bfm} H^{\le\bfr} \ot H^{\le\bfs}  = (H\ot H)^{<\bfn +\bfm}.
\end{aligned}
\end{equation}
Similarly, $(H\ot H)^{<\bfn}(H\ot H)^{\le \bfm} \subseteq (H\ot H)^{<\bfn +\bfm}$.

To prove the Lemma, let us put 
\[
\Sigma(\bfn):= \sum_{\bfm + \bfm' = \bfn} e_{\bfm}\ot e_{\bfm'}\,.
\]
Our goal is to show that $\D(e_{\bfn}) \in \Sigma(\bfn) + (H\ot H)^{<\bfn}$.
This certainly holds for $\bfn =\bfzero$,
since $e_{\bfzero} =1$ and $\D(1) = 1\ot 1 = \Sigma(\bfzero)$.
Now assume that $\bfn \neq \bfzero$ and proceed by induction on $|\supp \bfn|$.
If $|\supp \bfn| = 1$, then
$\bfn = n\d_{\lambda}$ and $e_{\bfn} = e_{n\d_\lambda} = \frac{e_{\lambda}^{n}}{n!}$
for some $\lambda \in \Lambda$, $n \in \ZZ_{>0}$\,.
By Lemma \ref{comH}, we know that
$\D(e_{\lambda}) \in 1\ot e_{\lambda} + e_{\lambda} \ot 1 + 
\sum_{i=1}^{|\lambda| -1} H_{i}\ot H_{|\lambda| - i}$.
The summand $H_{i}\ot H_{|\lambda| - i}$ is spanned by the tensors 
$e_{\bfj} \ot e_{\bfj}$ with $|\bfi| \le i$ and $|\bfj| \le |\lambda| - i$ by Lemma \ref{BasisM}; so $|\bfi + \bfj| \le |\lambda|$.
In addition, $|\bfi|, |\bfj| < |\lambda|$. Thus, if $\mu \geq \lambda$, then $\bfi(\mu) = \bfj(\mu) = 0$,
since $|\lambda| \leq |\mu|$. This shows that $\bfi+\bfj < \delta_{\lambda}$\,.
Therefore, $\sum_{i=1}^{|\lambda| -1} H_{i}\ot H_{|\lambda| - i} \subseteq (H\ot H)^{<\delta_{\lambda}}$ and hence
$\D(e_{\lambda}) \in 1\ot e_{\lambda} + e_{\lambda} \ot 1 + (H\ot H)^{<\delta_{\lambda}}$.
Now we compute:
\[
\begin{aligned}
\D(e_{n\d_{\lambda}}) &= \frac{1}{n!} \D(e_\lambda)^n 
\in\frac{1}{n!} \big(1\ot e_{\lambda} + e_{\lambda}\ot 1 + (H\ot H)^{<\delta_{\lambda}} \big)^{n}\\
&\subseteq \frac{1}{n!} (1\ot e_{\lambda} + e_{\lambda}\ot 1)^{n} + (H\ot H)^{<n\delta_{\lambda}},
\end{aligned}
\]
where the last inclusion
is obtained by expanding the product and using (\ref{E:prelim}) on all but the first summand.
The binomial theorem further gives $\frac{1}{n!} (1\ot e_{\lambda} + e_{\lambda}\ot 1)^{n} = 
\sum_{i=0}^{n} \frac{1}{i!(n-i)!} \, e_{\lambda}^{i}\ot e_{\lambda}^{n-i} 
= \sum_{i=0}^{n} e_{i \d_\lambda} \ot e_{(n-i) \d_\lambda} = \Sigma(n\d_\lambda)$\,. Therefore, 
\[
\D(e_{n\d_{\lambda}})
\in \Sigma(n\d_\lambda) + (H\ot H)^{<n\delta_{\lambda}}.
\]

For the inductive step, let $\bfn \in M$ with $|\supp \bfn| > 1$. Put $\mu:= \max \supp \bfn$, $n:= \bfn(\mu)$,
and $\bfn':= \bfn - n\delta_{\mu}$\,.
Thus, $\supp\bfn' = \supp\bfn \setminus \{\mu\}$
and hence, by induction,
$\D(e_{\bfn'}) \in \Sigma(\bfn') + (H\ot H)^{<\bfn'}$. Furthermore,
\[
e_{\bfn} = \prod_{\lambda\in \supp\bfn}^{<}\frac{1}{\bfn(\lambda)!} e_{\lambda}^{\bfn(\lambda)} 
= \Big( \prod_{\lambda\in \supp\bfn,\lambda \neq \mu}^{<}\frac{1}{\bfn(\lambda)!}\, e_{\lambda}^{\bfn(\lambda)}\Big) 
\frac{e_\mu^n}{n!}  = e_{\bfn'}\, e_{n\d_\mu}  \,.
\]
It follows that
\[
\D(e_{\bfn}) = \D(e_{\bfn'})\D(e_{n\d_\mu}) \in \big(\Sigma(\bfn') + (H\ot H)^{<\bfn'} \big)
\big(\Sigma(n\d_\mu) + (H\ot H)^{<n\delta_{\mu}} \big)
\]
By \eqref{E:prelim}, 
$(H\ot H)^{<\bfn'} \Sigma(n\d_\mu)
\subseteq (H\ot H)^{<\bfn'}(H\ot H)^{\leq n\delta_{\mu}}\subseteq (H\ot H)^{<\bfn}$
as well as
$\Sigma(\bfn')(H\ot H)^{<n\delta_{\mu}}
\subseteq (H\ot H)^{<\bfn}$ and $(H\ot H)^{<\bfn'}(H\ot H)^{<n\delta_{\mu}}\subseteq (H\ot H)^{<\bfn}$. Finally,
$\Sigma(\bfn')\Sigma(n\d_\mu) = \Sigma(\bfn)$. This completes the proof.
\end{proof}


\section{The convolution algebra}
\label{S:ConvAlg}

In this section, we fix an arbitrary $\k$-algebra $R$ (associative, with $1$) and consider
the convolution algebra $\Hom_{\k}(H,R)$; this is a $\k$-algebra with convolution $*$ as multiplication:
\[
(f * g)(h) = f(h_{1}) g(h_{2}) \qquad (f,g \in \Hom_{\k}(H,R), h\in H).
\]
We continue working under the standing assumptions that the Hopf algebra 
$H$ is connected and $\Char \k = 0$.

\subsection{Minimal support elements}

Let $M = \ZZ_{+}^{(\Lambda)}$ be the monoid of Section~\ref{monoidM}
and let $\mathscr{M}$ be the PBW basis of $H$ as in Section~\ref{SS:PBW};
so  $\mathscr{M}\cong M$ as sets via $e_{\bfn} \leftrightarrow \bfn$.
Since every $f \in \Hom_\k(H,R)$ is determined by its values on the basis $\mathscr{M}$,
which can be arbitrarily assigned elements of $R$, we have a bijection
$\Phi \colon \Hom_{\k}(H,R) \longiso R^M$,  the set of all functions $s \colon M \to R$;
explicitly,
\begin{equation}
\label{E:Phi}
(\Phi f)(\bfn) = f(e_{\bfn}) \qquad (f \in \Hom_\k(H,R), \bfn\in M).
\end{equation}

Now assume that $M$ is equipped with the well-order $\le$ of
Lemma~\ref{OrderM}. Note that the support
$\supp \Phi f = \{ \bfn \in M \mid f(e_{\bfn}) \neq 0 \}$
need not be finite if $f \neq 0$, but it does have a smallest element for the well-order $\le$\,. 
For $0 \neq f \in \Hom_\k(H,R)$, we may therefore define
\begin{equation}
\label{E:min}
\bff:= \min\supp\Phi f \in M
\qquad\text{and}\qquad
f_{\min}:= f(e_{\bff}) \in R \setminus \{ 0 \} .
\end{equation}

\begin{lem}
\label{Tech2}
Let $0 \neq f,g \in \Hom_\k(H,R)$ and define $\bff, \bfg \in M$ as in \eqref{E:min}. Then:
\begin{enumerate}
\item 
$(f\ast g)(e_{\bfn}) = 0$ for $\bfn < \bff + \bfg$ and $(f\ast g)(e_{\bff + \bfg}) = f(e_{\bff})\,g(e_{\bfg})$\,.
\item 
If $f(e_{\bff})\,g(e_{\bfg}) \neq 0$, then $f*g \neq 0$ and $(f*g)_{\min} = f(e_{\bff})\,g(e_{\bfg})$\,.
\end{enumerate}
\end{lem}

\begin{proof}
Since part (b) is clear from (a), we will focus on (a).

Note that $f$ vanishes on $H_{}^{<\bff}$, the subspace of $H$ that is generated by the basis elements
$e_{\bfn}$ with $\bfn <\bff$; similarly, $g(H_{}^{<\bfg}) = \{ 0\}$.
For any $e_{\bfi}\ot e_{\bfj} \in (H\ot H)^{<\bff+\bfg}$, we have either $\bfi < \bff$ or $\bfj < \bfg$,
because otherwise $\bfi + \bfj \ge \bff + \bfg$
by Lemma~\ref{OrderM}(b). Hence, $f(e_{\bfi})g(e_{\bfj}) = 0$ and so
\begin{equation}
\label{E:Tech2}
m\circ (f\ot g)(H\ot H)^{< \bff + \bfg} = 0.
\end{equation}

Now let $\bfn\in M$ be such that $\bfn \le \bff + \bfg$.
By Lemma \ref{Tech1},
\[
\D(e_{\bfn}) \in \sum_{\bfm +\bfm' = \bfn} e_{\bfm}\ot e_{\bfm'} + (H\ot H)^{<\bfn}.
\]
Equation \eqref{E:Tech2} gives us that $(f * g)(e_{\bfn}) = 0$ if $\bfn < \bff + \bfg$, and
\[
(f*g)(e_{\bfn}) = 
\sum_{\bfm+\bfm' = \bfn} f(e_{\bfm})g( e_{\bfm'})
\]
if $\bfn = \bff + \bfg$. Consider $(\bfm,\bfm') \in M\times M$ with $\bfm +\bfm' = \bfn$ as in the sum above.
If $\bfm < \bff$ or $\bfm' < \bfg$, then $f(e_{\bfm})g( e_{\bfm'}) = 0$. Therefore, the only contribution
to the sum comes from the pair $(\bfm,\bfm') = (\bff,\bfg)$, proving the lemma.
\end{proof}

\subsection{Subrings of the convolution algebra}
\label{Subrings}

The unit map $u = u_H \colon \k \to H$ gives rise to the following algebra map:
\[
u^* \colon \Hom_\k(H,R) \onto R\,, \quad f \mapsto f(1).
\]
The theorem below is an adaptation of \cite[Lemma 16]{LNY20}.

\begin{thm}
\label{T:Tech3}
Assume that $\ch\k = 0$.
Let $R$ be a $\k$-algebra, let $H$ be a connected Hopf $\k$-algebra, and
let $S \subseteq \Hom_\k(H,R)$ be a subring such that $u^*(S) = R$\,.
If $R$ is prime (semiprime, a domain), then so is $S$.
\end{thm}

\begin{proof}
First assume that $R$ is prime. Let $0 \neq s,t \in S$ be given and let
$s_{\min}, t_{\min} \in R \setminus \{ 0 \}$ be as in \eqref{E:min}. Then $0 \neq s_{\min}rt_{\min}$ for some $r \in R$.
By assumption, there exists an element $u\in S$ such that $u(1) = r$. Since $1 = e_\bfzero$, we
evidently have $\bfzero = \min \supp \Phi u$ and $u_{min} = r$.
It now follows from Lemma \ref{Tech2} that $s*u*t\neq 0$ and
$(s*u*t)_{min}= s_{min}rt_{min}$\,.
This proves that $S$ is prime. 

For the assertions where $R$ is semiprime or a domain, take $s = t$ or $r =1$, respectively.
\end{proof}

Of course, Theorem~\ref{T:Tech3} applies with $S = \Hom_{\k}(H,R)$. Thus, we obtain
the following corollary.

\begin{cor}
\label{C:Tech3}
Let $H$ be a connected Hopf algebra over a field $\k$ of characteristic $0$ and let $R$ be a $\k$-algebra that is prime (semiprime, a domain), then so is the convolution algebra $\Hom_\k(H,R)$.
\end{cor}

\subsection{Proof of Theorem~\ref{Thm}}

We are now ready to give the proof of Theorem~\ref{Thm}, which is also a consequence of
Theorem~\ref{T:Tech3}.

Let $A$ be an $H$-module algebra and let $I$ be an arbitrary ideal of $A$.
The core $(I:H)$ is the kernel of the map
$\rho \colon A \longrightarrow \Hom_{\k}(H,A/I)$ that is defined by
$\rho(a) = (h \mapsto h.a + I)$. So $\rho(A) \cong A/(I:H)$ as rings. Note that
the composite map $u^* \circ \rho \colon A \to \Hom_{\k}(H,A/I) \to A/I$ is 
the canonical epimorphism, $a \mapsto a+I$ $(a \in A)$. So $u^*(\rho(A)) = A/I$.
Therefore, Theorem~\ref{T:Tech3} applies with $S = \rho(A)$ and $R = A/I$.
We obtain that when $A/I$ is prime (semiprime, a domain), then so is $\rho(A)$.
Since $\rho(A) \cong A/(I:H)$, this
is equivalent to the statement of Theorem~\ref{Thm}. \qed

\begin{ack}
The results presented in the foregoing are to form part of the author's Ph.D.~dissertation at Temple 
University.
He would like to thank his adviser, Martin Lorenz, for his guidance and many discussions on the 
subject of this paper.
\end{ack}


\bibliographystyle{plain}
\bibliography{bibliography}

\def\cprime{$'$}
\begin{thebibliography}{1}

\bibitem{wC92}
William Chin.
\newblock Actions of solvable algebraic groups on noncommutative rings.
\newblock In {\em Azumaya algebras, actions, and modules (Bloomington, IN,
  1990)}, volume 124 of {\em Contemp. Math.}, pages 29--38. Amer. Math. Soc.,
  Providence, RI, 1992.

\bibitem{jD96}
Jacques Dixmier.
\newblock {\em Enveloping algebras}, volume~11 of {\em Graduate Studies in
  Mathematics}.
\newblock American Mathematical Society, Providence, RI, 1996.
\newblock Revised reprint of the 1977 translation.

\bibitem{mL08}
Martin Lorenz.
\newblock Group actions and rational ideals.
\newblock {\em Algebra Number Theory}, 2(4):467--499, 2008.

\bibitem{mL18}
Martin Lorenz.
\newblock {\em A Tour of Representation Theory}, volume 193 of {\em Graduate
  Studies in Mathematics}.
\newblock American Mathematical Society, Providence, RI, 2018.

\bibitem{LNY20}
Martin Lorenz, Bach Nguyen, and Ramy Yammine.
\newblock Actions of cocommutative {H}opf algebras.
\newblock {\em J. Algebra}, 546:703--722, 2020.

\bibitem{sM93}
Susan Montgomery.
\newblock {\em Hopf algebras and their actions on rings}, volume~82 of {\em
  CBMS Regional Conference Series in Mathematics}.
\newblock Published for the Conference Board of the Mathematical Sciences,
  Washington, DC, 1993.

\bibitem{dR12}
David~E. Radford.
\newblock {\em Hopf algebras}, volume~49 of {\em Series on Knots and
  Everything}.
\newblock World Scientific Publishing Co. Pte. Ltd., Hackensack, NJ, 2012.

\bibitem{dLySgZxx}
G.-S. Zhou, Y.~Shen, and D.-M. Lu.
\newblock The structure of connected (graded) {H}opf algebras.
\newblock arXiv:1904.01918 [math.RA], 2019.

\bibitem{gZ13}
Guangbin Zhuang.
\newblock Properties of pointed and connected {H}opf algebras of finite
  {G}elfand-{K}irillov dimension.
\newblock {\em J. Lond. Math. Soc. (2)}, 87(3):877--898, 2013.

\end{thebibliography}


\end{document}